\documentclass[12pt]{article}

\usepackage{amssymb}
\usepackage{pdfpages}
\usepackage{graphicx}
\usepackage{floatrow}
\usepackage{amsmath}
\usepackage{amsfonts}
\usepackage[colorlinks,linkcolor=red,citecolor=blue]{hyperref}
\hypersetup{
  colorlinks=true,
  urlcolor=blue}
\newtheorem{theorem}{Theorem}
\newtheorem{corollary}[theorem]{Corollary}
\newtheorem{definition}{Definition}
\newtheorem{observation}{Observation}

\newtheorem{lemma}[theorem]{Lemma}

\newtheorem{claim}{Claim}

\newenvironment{proof}{\noindent {\bf Proof.}}{\hfill\rule{3mm}{3mm}\par\medskip}

\newtheorem{prelem}{{\bf Theorem}}

\newenvironment{oldtheorem}{\begin{prelem}{\hspace{-0.5
em}{\bf}}}{\end{prelem}}

\newtheorem{prelemc}{{\bf Conjecture}}

\title{On the algorithmic complexity of finding hamiltonian cycles in special classes  of planar  cubic graphs}
\author{\sc Behrooz Bagheri Gh.${}^{a}$\footnote{E-mail: behrooz@ac.tuwien.ac.at.}, Tomas Feder\footnote{268 Waverley St., Palo Alto, CA 94301, E-mail: tomas@theory.stanford.edu.},  Herbert Fleischner${}^{a}$\footnote{E-mail: fleisch@dbai.tuwien.ac.at.}, \\ Carlos Subi\footnote{E-mail: carlos.subi@hotmail.com.}}
\date{}

\begin{document}
\maketitle
  \vspace{-1cm}
  \begin{center}
  	$a$
  	{\small \it Algorithms and Complexity Group}\\
  	{\small \it Vienna University of Technology}\\
  	{\small  \it Favoritenstrasse 9-11,}
  	\vspace*{5mm}
  	{\small \it 1040 Vienna, Austria }\footnote{Research supported in part by FWF Project P27615-N25.} \\
 
  \end{center}

\begin{abstract}
It is a well-known fact that hamiltonicity in planar cubic graphs is an 
NP-complete problem. This implies that the existence of an $A-$trail in 
plane eulerian graphs is also an NP-complete problem even if restricted 
to planar $3-$connected eulerian graphs.
In this paper we deal with hamiltonicity in planar cubic graphs $G$ 
having a facial $2-$factor $\mathcal{Q}$ via (quasi) spanning trees of 
faces in $G/\mathcal{Q}$ and study the algorithmic complexity of finding 
such (quasi) spanning trees of faces. We show, in particular, that 
if Barnette's Conjecture is false, then hamiltonicity in $3-$connected 
planar cubic bipartite graphs is an NP-complete problem.\\
\\
%
\end{abstract}

\section{Introduction and Preliminary Discussion}

Our joint paper~\cite{Bagheri} can be considered as the 
point of departure for the subsequent discussion and results of this 
paper. Next, we make a few historical remarks.
In $1884$, Tait  conjectured that every cubic 
$3-$connected planar graph is 
hamiltonian~\cite{Tait}. 
And Tait knew that the validity of his conjecture would yield a simple 
proof of the Four Color Conjecture. On the other hand, the Petersen graph 
is the smallest non-planar $3-$connected cubic graph which is not hamiltonian,~\cite{Petersen}.
Tait's Conjecture
 was disproved by Tutte in $1946$. However, none of the known counterexamples of Tait's Conjecture is 
bipartite. Tutte himself conjectured that every cubic $3-$connected 
bipartite graph is hamiltonian~\cite{Tutte1971}, but this was shown to be 
false by the construction of a counterexample, 
the Horton graph~\cite{Horton}.  Barnette  proposed a 
combination of Tait's and Tutte's Conjectures implying that every counterexample 
to Tait's conjecture is non-bipartite.\\

\medskip \noindent
{\bf Barnette's Conjecture}~\cite{Barenette}
{\it Every $3-$connected cubic planar bipartite graph is hamiltonian.}\\

Holton, Manvel and McKay showed in~\cite{Holton}  that Barnette's Conjecture holds true for graphs with up to
$64$ vertices. The conjecture 
also holds for the infinite family of graphs where all faces are either quadrilaterals or  hexagons, as shown by
Goodey~\cite{Goodey}. However, it is NP-complete to decide whether
a $2-$connected cubic planar bipartite graph  is hamiltonian~\cite{Takanori}. 

For a more detailed account of the early development of hamiltonian graph theory we refer the interested reader 
to~\cite{Biggs}.

As for the terminology used in this paper we follow~\cite{Bondy} unless stated explicitly otherwise.
In particular, the subset $E(v)\subseteq E(G)$ denotes the set of edges incident to $v\in V(G)$.

\medskip \noindent
We repeat some definitions stated already in~\cite{Bagheri} in order to make life easier for the reader.

\begin{definition}
A cubic graph $G$ is {\sf cyclically $k-$edge-connected} if at least $k$ edges must be removed to disconnect $G$ either into two components 
 each of which contains a cycle provided $G$ contains two disjoint cycles, or else into two non-trivial components. The {\sf cyclic edge-connectivity} of $G$ is the maximum $k$ such that $G$ is cyclically $k-$edge-connected, denoted $\kappa{'}_c(G)$.
\end{definition}
\begin{definition}
Let $C$ be a cycle in a plane graph $H$. The cycle $C$ divides 
the plane into two disjoint open domains. The  {\sf interior} $(${\sf exterior}$)$ of $C$ 
is the bounded $($unbounded$)$ domain and is denoted by 
$int(C)\ (ext(C))$. 
By treating parallel edges as a single edge,
we say a cycle $C^{'}$ is {\sf inside} of 
$C$ if $int(C^{'})\subseteq int(C)$.
Moreover,  
a cycle $C$ is said 
{\sf to contain a vertex $v$ inside (outside)}   if $v\in int(C)\ (v \in ext(C))$. Finally,  $C$ is said to be a  {\sf separating} cycle in $H$ if 
$int(C)\cap V(H)\neq \emptyset \neq ext(C)\cap V(H)$.
\end{definition}
\begin{remark}
\begin{description}
\item[1.]
Two edges $e=xy$ and $e^{'}=xy$ are called {\sf parallel edge}s
if the digon $D$ defined by $e$ and $e^{'}$ has no vertices inside.
If two different triangles $T_1$ and $T_2$ have an edge  in common, then 
they have no other edge in common (because of our understanding that 
parallel edges are treated as a single edge), unless there is 
$e_i=xy\in E(T_i),\ i=1,2$, such that $\langle e_1,e_2\rangle$ defines a digon with 
some vertex  inside.
\item[2.] 
Given a $2-$connected plane  graph, we do not distinguish between faces and their face boundaries.
Observe that in planar $3-$connected graphs $H$,
the face boundaries are independent from any actual 
embedding of $H$ in the plane or sphere.
\end{description}
\end{remark}
\begin{definition}
Given a graph $H$ and a vertex $v$, a fixed sequence 
$\langle e_1,\ldots,e_{\deg(v)}\rangle$ of the edges in $E(v)$ 
is called a {\sf positive ordering} of $E(v)$ and is denoted by 
$O^{+}(v)$. If  $H$ is imbedded in some surface, one such $O^{+}(v)$
is given by the counterclockwise cyclic ordering of the edges incident 
to $v$.
\end{definition}

\begin{definition}\label{DEF:A-trail}
Let $H$ be an eulerian graph with a given positive ordering $O^{+}(v)$ 
for each vertex $v\in V(G)$. An eulerian trail $L$ is an {\sf $A-$trail} 
if  $\{e_i,e_j\}$ being a pair of consecutive edges  in $L$ 
incident to $v$ implies 
$j=i\pm1$ $({\rm mod}\ \deg(v))$. 
-- As a consequence, in an $A-$trail in a $2-$connected plane graph any two consecutive edges belong to a face boundary.

An $A-$trail $L$ in an eulerian triangulation of the plane is called {\sf non-separating} if for every face boundary $T$  at least two edges of $E(T)$ are consecutive in 
$L$.  
\end{definition} 

An $A-$trail $L=e_1e_2\ldots e_m$ induces a vertex partition $V_L(H)=\{V_1,V_2\}$ on $V(H)$ as follows.  Consider  a 
$2-$face-coloring of $H$ with colors $1$ and $2$. For every vertex $v$ of $H$, $v\in V_i$ if and only if  there is $j\in \{1,\ldots,m-1\}$ such that $v\in V(e_j)\cap V(e_{j+1})$ and the face containing $e_j$ and $e_{j+1}$ in its boundary is colored $3-i$,  $i=1,2$.

\begin{oldtheorem}~$(${\rm\cite[Theorem~2]{Andersen95}}$)$\label{TH:A-Tr.NP-C} 
The problem of deciding whether a planar eulerian graph admits an 
$A-$trail is NP-complete, even for $3-$connected graphs having only $3-$cycles
and $4-$cycles as face boundaries.
\end{oldtheorem}
In contrast,  Andersen et al. in~\cite{Andersen98} gave a polynomial algorithm for 
finding $A-$trails in simple $2-$connected outerplane eulerian graphs.

\begin{definition}\label{DEF:Radial}
Suppose $H$ is a $2-$connected plane  graph. Let  $\mathcal{F}(H)$ be 
the set of faces of $H$. The {\sf radial graph} of $H$ denoted by 
$\mathcal{R}(H)$ is a bipartite graph
with the vertex bipartition $\{V(H),\mathcal{F}(H)\}$ 
such that $xf\in E(\mathcal{R}(H))$ if and only if $x$ is a vertex in the boundary of 
$F\in \mathcal{F}(H)$ corresponding to $f\in V(\mathcal{R}(H))$.

Let $U\subseteq V(H)$ and let $\mathcal{T}\subset \mathcal{F}(H)$ be a 
set of bounded faces of $H$. 
The {\sf restricted radial graph} 
$\mathcal{R}(U,\mathcal{T})\subset \mathcal{R}(H)$ is defined  by $\mathcal{R}(U,\mathcal{T})=\langle U\cup \mathcal{T}\rangle_{\mathcal{R}(H)}$.
\end{definition}

\begin{definition}\label{DEF:Leapfrog}
Let $G$ be a $2-$connected plane graph and  $v$ be a vertex of $G$ with $\deg(v)\ge 3$. Also assume that a sequence 
$\langle e_1,\ldots,e_{\deg(v)}\rangle$, $e_i=u_iv,\ i=1,\ldots,\deg(v)$ 
is given by the counterclockwise cyclic ordering of the edges incident to $v$. 
\begin{description}

\item[$(i)$]
  A {\sf truncation} of $v$ is the process of replacing $v$ with a 
cycle $C_v=v_1\ldots v_{\deg(v)}v_1$ and replacing $e_i=u_iv$ with 
$e_i^{'}=u_iv_i$, for $i=1,\ldots,\deg(v)$ in such a way
that the result is a plane graph again. A plane graph obtained from 
$G$ by truncating  all vertices of $G$ is called {\sf  the truncation of $G$} 
and denoted by $Tr(G)$.
\item[$(ii)$]
For a  plane graph $G$ let 
$G^{*}$ denote   the dual of $G$. 
The {\sf leapfrog extension} $Lf(G)$ of a plane 
graph $G$ is $(G\cup \mathcal{R}(G))^{*}$. In the case of cubic $G$,
it can be viewed as obtained from $G$  by replacing every $v\in V(G)$ by
a hexagon $C_6(v)$, with $C_6(v)$ and $C_6(w)$ sharing an edge if and 
only if $vw\in E(G)$; and these hexagons are faces of $Lf(G)$. 
\end{description}
\end{definition}

\begin{oldtheorem}~$(${\rm\cite[Theorem~23]{FleischnerEnvelope}}$)$\label{TH:Lf.NP-C} 
The question of whether the leapfrog extension of a plane cubic graph 
with multiple edges is hamiltonian is NP-complete.
\end{oldtheorem}

We note in passing that we call leapfrog extension, 
what is called in other papers  vertex envelope, or leapfrog construction, or leapfrog operation, or  leapfrog transformation (see e.g.~\cite{FleischnerEnvelope,Fowler,Kardos,Yoshida}). 

\begin{definition}
Let $H$ be a  $2-$connected plane graph, let $U\subset V(H)$, and let $\mathcal{T}\subset \mathcal{F}(H)$ be a set of bounded faces. We define a subgraph $H_{\mathcal{T}}$ of $H$  by 
$H_{\mathcal{T}}=\big\langle \cup_{F\in \mathcal{T}}E(F)\big\rangle$.
If for every $x\in V(H)\setminus U$, $\big|\big\{F\in \mathcal{T}\ :\ x\in V(F)\big\}\big|=\frac{1}{2} \deg_H(x)$
and if $\mathcal{R}(U,\mathcal{T})$ is a tree, then we call $\mathcal{T}$
a {\sf quasi spanning tree of faces} of $H$, and the vertices in 
$U\ (V(H)\setminus U)$ are called {\sf proper} $(${\sf quasi}$)$ vertices.
If $U=V(H)$, then $\mathcal{T}$ is called a {\sf spanning tree of faces}.
\end{definition}

If a plane graph has a face-coloring with color set $X$, the faces of color $x\in X$ will then be called $x-$faces.

\begin{observation}\label{obs:A-trail}
We observe that if $H$ is a  plane eulerian graph with $\delta(H)\ge 4$
having an $A-$trail $L$, then  $L$ defines 
uniquely a quasi spanning tree of faces as follows. 
Starting with a 
$2-$face-coloring of $H$ with colors $1$ and $2$, suppose the outer face 
of $H$ is colored $1$.    Let $V_L(H)=\{V_1, V_2\}$ be the  
partition of $V(H)$ induced by $L$.
Now, the set of all $2-$faces  defines a quasi spanning tree of
faces $\mathcal{T}$ with $V_1$ being the set of all quasi vertices of  
$\mathcal{T}$.
Conversely, a $($quasi$)$ spanning tree of faces $\mathcal{T}$ defines
uniquely an $A-$trail in $H_\mathcal{T}$.
\end{observation}

The aforementioned relation between the concepts of $A-$trail  and 
(quasi) spanning tree of faces is not a coincidence. In fact, it had been shown
(\cite[pp.~$VI.112-VI.113$]{Fleischner})
that 

$\bullet$ {\it Barnette's Conjecture is true if and only if every simple $3-$connected eulerian 
triangulation of the plane admits an $A-$trail.}\\ 

We point out, however, that the concept of (quasi) spanning tree of faces
is a somewhat more general tool to deal with hamiltonian cycles in
plane graphs, than the concept of $A-$trails. 

We are thus focusing our considerations below on the 
complexity of the existence of $A-$trails and (quasi) spanning trees 
of faces in plane (eulerian) graphs.

Parts of this paper are the result of extracting
some results and their proofs of~\cite{Feder} which appear correct to 
all four of us; they have not been published yet.
Moreover, we relate some of the results of this paper to the theory of 
$A-$trails, as developed in~\cite{Fleischner}.

\medskip \noindent
In the sequel

\medskip \noindent
{\it $G$ always denotes a $3-$connected cubic plane graph 
having a facial $2-$factor $\mathcal{Q}$ $($i.e., a $2-$factor whose cycles are face 
boundaries of $G)$; we denote the set of face boundaries of $G$ not in 
$\mathcal{Q}$ by $\mathcal{Q}^{c}$.
 In general, when we say that $F$ is an $\mathcal{X}-$face, we mean that 
  $F\in \mathcal{X}$.
 Let $H$ always denote the reduced graph obtained from $G$ by contracting  the $\mathcal{Q}-$faces 
  to single vertices; i.e., $H=G/\mathcal{Q}$.
  \hfill $(\bf H)$}\\

Next, we list several results of a preceding joint paper
which will be essential for the current paper.

\begin{oldtheorem}~$(${\rm\cite[Proposition~1]{Bagheri,Feder}}$)$
\label{PR:1}
Let $G,\mathcal{Q},$ and $H$  be as stated in $(\bf H)$.
  The reduced
graph $H$ has a quasi spanning tree of faces, $\mathcal{T}$, with
 the external face not 
in $\mathcal{T}$ if and only if $G$
has a hamiltonian cycle $C$ with  the external $\mathcal{Q}^{c}-$face outside of $C$, 
with all $\mathcal{Q}-$faces corresponding to proper vertices  
of $\mathcal{T}$
inside of $C$, with 
all $\mathcal{Q}-$faces corresponding to quasi vertices 
of $\mathcal{T}$
outside of $C$, and such that no two $\mathcal{Q}^{c}-$faces sharing an edge are both 
inside of $C$.
\end{oldtheorem}

\begin{oldtheorem}~$(${\rm\cite[Corollary~7]{Bagheri}}$)$\label{cor:4con-eulerian-triangulation}
Every simple $4-$connected eulerian triangulation of the plane has a 
quasi  spanning tree of faces.
\end{oldtheorem}

\section{Polynomial and NP-Complete Problems}\label{Sec:Complexity}

In proving   Propositions $3$ and $6$, and Theorem $11$ 
in~\cite{Bagheri}, we used implicitly some 
algorithms to construct a (quasi) spanning tree of faces. In all of them, 
we find some triangular face such that the  graph 
resulting from the 
contraction of this face, still satisfies the 
hypothesis of the respective result.
 By repeating this 
process, finally the contracted faces together  with a 
special face form a (quasi) spanning tree of faces.
 Note that it is possible to identify the contractible faces in linear time, since every simple plane graph has $O(n)$ faces, where $n$ is the order of graph.
 Therefore, our 
algorithms for finding a quasi spanning tree of faces   in~\cite{Bagheri} 
are  polynomial.

We show next that one can decide in polynomial time whether the reduced graph $H$ has a spanning tree of
faces which are either digons or triangles. The result easily extends to the
case of a spanning tree of faces where all but a constant number of faces are either digons
or triangles.

{\bf The Spanning Tree Parity Problem:} {\it 
Given a graph $G$ and a collection of pairs of edges,
$\{\{e_i,f_i\}\ |\ i=1,\ldots, k\}$. {\sf The Spanning Tree Parity Problem} 
asks whether $G$ has a spanning tree $T$ satisfying $|\{e_i,f_i\}\cap E(T)|\in \{0,2\}$, for each $i=1,\ldots,k$.}

Note that the Spanning Tree Parity Problem is 
solvable in polynomial time (see~{\rm\cite{Gabow,Lovasz}}).

\begin{theorem}\label{TH:3}
Let $G,\mathcal{Q},$ and $H$  be as stated in $(\bf H)$.
 Let $D$ be the set
 of faces in $H$ such that
all faces in $\mathcal{D}$ are either digons or triangles. Then we can decide in polynomial time whether $H$ has a spanning
tree of faces in $D$, giving a hamiltonian cycle for $G$, by a spanning tree
parity algorithm.
\end{theorem}
\begin{proof}{ Construct a graph $H^{'}$ related to $H$ as follows. $V(H^{'})=V(H)$. If $xyx$ is a digon in $\mathcal{D}$, then let  $xy$ be an edge in $H^{'}$ . If $xyzx$ is a triangle
in $\mathcal{D}$, then put edges $xy$ and $yz$ in $H^{'}$
(the naming of the  vertices of the triangle with the symbols $x,y,z$ is arbitrary but fixed). A spanning tree of faces in $\mathcal{D}$ for H then corresponds
to a  spanning tree in $H^{'}$ which must contain either both or none of 
 the edges $xy$ and
$yz$ corresponding to  the triangle $xyzx$ in $\mathcal{D}$. Thus, these conditions on pairs of edges in $H^{'}$ 
transform the problem of finding a spanning tree of faces  
 in $\mathcal{D}$ for $H$ equivalently in polynomial time into a Spanning Tree Parity Problem in $H^{'}$.
}\end{proof}

If $\mathcal{D}$ contains faces with four or more sides, say a face $xyztx$, then we could include three
edges linking these four vertices, say $xy,\ yz,$ and $zt$, and require that a spanning tree  must contain
either all three or none of these three edges. Such a spanning tree triarity problem, as we
shall  see later, turns out to be NP-complete.

The following is immediate from Theorem~\ref{PR:1}.
 
\begin{theorem}~$(${\rm\cite[Proposition~3]{Feder}}$)$\label{PR:3} Let $G$ be a $3-$connected cubic plane bipartite graph whose faces are $3-$colored with color set $\{1,2,3\}$ and the outer face of $G$ is a $3-$face. Then the following statements are equivalent.

\begin{description}

\item[$(i)$]  $G$ has a hamiltonian cycle $C$ with the $2-$faces  lying inside of $C$, the $3-$faces lying outside of $C$, and $1-$faces on either side;
\item[$(ii)$] the reduced graph $H$ obtained by contracting  the
$1-$faces has an $A-$trail;
\item[$(iii)$] the reduced graph $H^{'}$ obtained by
contracting  the $2-$faces has a spanning tree of $1-$faces;
\item[$(iv)$] the reduced graph $H^{''}$
obtained by contracting  the $3-$faces has a spanning tree of $1-$faces.
\end{description}
\end{theorem}

\begin{proof} $(i)\Rightarrow(ii):$ Let $T_C$ be a 
closed trail in $H$ induced by 
hamiltonian cycle $C$ of $G$. The 
closed trail $T_C$ is an eulerian trail, otherwise there are two  faces of $G$ with two different colors $2$ and $3$ lying on one side of $C$. Since all $2-$faces ($3-$faces) of $G$ are lying inside (outside) of $C$, for every $1-$face $F_1$ of $G$ we conclude that $E(F_1)\cap E(C)$ is a matching. Thus, $T_C$ is an $A-$trail.\\
$(ii)\Rightarrow(i):$ It is easy to see that any $A-$trail of $H$ can be transformed into a hamiltonian cycle $C$ of $G$ with the $2-$faces  lying inside of $C$, the $3-$faces lying outside of $C$, and $1-$faces lying on either side.\\
$(i)\Rightarrow(iii):$  Let $U= V(H^{'})$ be the vertex set corresponding to the $2-$faces. Also, let $\mathcal{T}$ be the set of  $1-$faces
of $H^{'}$ corresponding to the $1-$faces in $int(C)$.

Observe that $G_{int}:=C\cup int(C)$
is a spanning outerplane subgraph of $G$, and that the weak dual
(the subgraph of the dual graph whose vertices correspond to the bounded faces) of 
$G_{int}$ is a tree (see~\cite{Fleischner}). Therefore, 
$H^{'}_{int}\subset H^{'}$ being the reduced graph of $G_{int}$ 
after contracting the $2-$faces, is a spanning tree of faces in $H^{'}$.
\\
$(iii)\Rightarrow(i):$   Suppose $H^{'}$ has a spanning tree of 
 $1-$faces $\mathcal{T}$.   
 Then $H^{'}_{\mathcal{T}}$ has a unique $A-$trail which can be transformed    into a hamiltonian cycle $C$ of $G$ such that
 the $2-$faces 
 (corresponding to $V(H^{'})$) lie in $int(C)$ and
the  corresponding $3-$faces   lie in $ext(C)$. 

The equivalence of $(i)$ and $(iv)$ is established analogously by looking at $G_{ext}:=C\cup ext(C)$ which is also an outerplanar graph.
\end{proof}

An application of Theorems~\ref{TH:3} and~\ref{PR:3} 
yields the following.

\begin{corollary}\label{CR:2} Let $G$ be a cubic plane bipartite graph 
with a $3-$face-coloring with color set $\{1,2,
3\}$, and let $H$ be the reduced graph obtained by contracting 
 the $1-$faces. Suppose all vertices of $H$ have degree $4$ or $6$.
Then one can decide in polynomial time whether $H$ has an $A-$trail
 which in turn yields a hamiltonian cycle in $G$.
\end{corollary}

\begin{proof}{ Let $H^{'}$ be the reduced graph of $G$ obtained by contracting the $2-$faces instead
of the $1-$faces. Then  each $1-$face of $G$ yields a 
digon or triangle in $H^{'}$. By Theorem~\ref{PR:3}, an
$A-$trail  in $H$ corresponds to a spanning tree of $1-$faces in $H^{'}$. Since all
$1-$faces  of $H^{'}$ are either digons or triangles, one can decide
in polynomial time by Theorem~\ref{TH:3}, whether such a spanning tree
of $1-$faces exists in $H^{'}$.
}\end{proof}

By Observation~\ref{obs:A-trail}, we have the following theorem 
in which we make use of the fact that an 
(eulerian) triangulation of the plane admits two interpretations, namely: as the dual of a plane cubic   (bipartite) graph, and as the contraction of a facial (even) $2-$factor $Q$ in $G$ whose faces in $Q^c$ are hexagons. Also note that a plane bipartite cubic graph  $G$ has a hamiltonian cycle if and only if the dual graph $G^*$ has a non-separating $A-$trail~\cite{Fleischner}.
\begin{theorem}\label{TH:Herbert}
	 Let $G$ be a  $3-$connected cubic planar bipartite graph and let $\mathcal{F}$ be the set of its faces. Let  
$\mathcal{Q}_{\mathcal{F}}$ be the facial $2-$factor of $Lf(G)$ 
corresponding to $\mathcal{F}$ and let the color classes of the $3-$face-coloring of $Lf(G)$ be denoted by $F_1$, $F_2$, and $F_3$ such that $F_3 = \mathcal{Q}_{\mathcal{F}}$, and thus $F_1$, $F_2$ translates into a $2-$face-coloring of $Lf(G)/\mathcal{Q}_{\mathcal{F}}$  denoting the corresponding sets of faces by $F1$, $F2$ and whose vertex set $($corresponding to $F_3)$ be denoted by $V_3$. Then the following is true.
\begin{description}
\item[$(1)$]
 $G^{*} = Lf(G)/\mathcal{Q}_{\mathcal{F}}.$
\item[$(2)$] $G$ is hamiltonian if and only if $Lf(G)$ has a hamiltonian cycle $C$ such that $int(C) = F_1\cup F_3^{'}$ and $ext(C) = F_2\cup F_3^{''}$ where $F_3 = F_3^{'}\dot{\cup} F_3^{''}$.\\
\\
Statement $(2)$ is equivalent to
\item[$(3)$] $G^{*}$ has a non-separating  $A-$trail if and only if
\begin{itemize}
\item[$(i)$]
  $Lf(G)/\mathcal{Q}_{\mathcal{F}}$ has a quasi spanning tree of faces containing all of F1 and where $V_3^{'}$ is its set of proper vertices and $V_3^{''}$ is its set of quasi vertices; and
\item[$(ii)$]
 $Lf(G)/\mathcal{Q}_{\mathcal{F}}$ has a quasi spanning tree of faces containing all of $F2$ and where $V_3^{''}$ is its set of proper vertices and $V_3^{'}$ is its set of quasi vertices.
$(V_3^{'}$ and $V_3^{''}$ are the vertex sets in  $Lf(G)/\mathcal{Q}_{\mathcal{F}}$ corresponding to $F_3^{'}$ and $F_3^{''}$, respectively.$)$ 
\end{itemize}
\end{description} 
\end{theorem}

\begin{proof}
By Definition~\ref{DEF:Leapfrog} and definition of the dual graph of a plane graph,  statement $(1)$ is true.

Assume $G$ has a hamiltonian cycle $C_0=e_1e_2\ldots e_n$ such that $e_i=v_iv_{i+1}$ for $i=1,\ldots,n-1$.
Let $e=v_iv_j\in E(G)$ be the edge  corresponding to $e^{'}\in E(C_6(v_i))\cap E(C_6(v_j))\subset E(Lf(G))$, for $1\le i\neq j\le n$ (see Definition~\ref{DEF:Leapfrog} $(ii)$
concerning $C_6(v_i)$).

Now we construct a hamiltonian cycle $C$ in $Lf(G)$ corresponding to $C_0$ as follows.
Begin with $C=e_1^{'}$. Consider the shortest path in $Lf(G)$ between  the endvertex of $e_1^{'}$ lying inside of $C_0$ and the endvertex of $e_2^{'}$ lying inside  of $C_0$, and add this path to $C$. Then add $e_2$ and add
the shortest path in $Lf(G)$ between  the endvertex of $e_2^{'}$  lying outside of $C_0$ and the endvertex of $e_3^{'}$ lying outside  of $C_0$, to the already constructed $C$. Add $e_3^{'}$ to $C$ and continue this algorithm such that at the $n-$th step you must add the  shortest path in $Lf(G)$ between  the endvertex of $e_n^{'}$  lying outside of $C_0$ and the endvertex of $e_1^{'}$ lying outside  of $C_0$ to the $C$.
It is easy to check that $C$ is a hamiltonian cycle in $Lf(G)$  
such that $int(C) = F_1\cup F_3^{'}$ and $ext(C) = F_2\cup F_3^{''}$ where $F_3 = F_3^{'}\dot{\cup} F_3^{''}$ (cf.~\cite[pp.~$VI.109-VI.110$]{Fleischner}).

Conversely, it is straightforward to see that a hamiltonian cycle in $Lf(G)$ as described yields a hamiltonian cycle in $G$.

The remainder of the proof follows from the paragraph preceding the statement of the theorem.
\end{proof}

Theorem~\ref{TH:Herbert} puts hamiltonicity in $G$ in a qualitative perspective of the algorithmic complexity regarding quasi spanning trees of faces of a special type in 
the reduced graph  of the leapfrog extension of $G$.
In fact, if $\mathcal{G}$ is a class of $3-$connected cubic planar bipartite graphs where hamiltonicity can be decided in polynomial time, then the same can be said regarding 
special types of
quasi spanning trees of faces in the reduced graphs of the leapfrog extensions of  $\mathcal{G}$ (as stated in the theorem). For, given a hamiltonian cycle $C_0$ in $G\in \mathcal{G}$, a
non-separating  $A-$trail $L_{C_0}$ in $G^{*}$ can be found in polynomial time which in turn yields a quasi spanning tree of faces in $Lf(G)/\mathcal{Q}_{\mathcal{F}}$ as described in $(3)\ (i)$ and  $(3)\ (ii)$, respectively,
also in polynomial time. Compare this with Theorem~\ref{TH:Lf.NP-C} and Theorem~\ref{cor:4con-eulerian-triangulation}.

We now establish several NP-completeness results.

The question of whether a $3-$connected planar cubic graph $G_0$ has a hamiltonian
cycle is NP-complete, as shown by Garey et al.~\cite{Garey}. 
Let $e =uv\in E(G_0)$. Then the question of whether $G_0$ has a hamiltonian 
cycle traversing this specified edge $e$, is also 
NP-complete. Let $G_0^{'}=G_0\setminus \{e\}$. Thus, the question
of whether $G_0^{'}$ has a hamiltonian path from $u$ to $v$ is also NP-complete.

\begin{theorem}~$(${\rm\cite[Theorem~4]{Feder}}$)$\label{TH:4} 
Let $G$ be a $3-$connected cubic planar bipartite graph.
Let a $3-$face-coloring with color set $\{1,2,
3\}$ be given, and let $H$ be the reduced graph obtained by contracting the $1-$faces. Suppose that the $2-$faces 
correspond in $H$ to quadrilaterals 
and the $3-$faces correspond in $H$ 
to digons. Then the question of whether $H$ has a spanning tree of $2-$faces
is NP-complete.
\end{theorem}

\begin{proof}{ We want to construct $G$ and $H$ as stated  in the theorem. To this end, 
let  $G_0^{'}$ be given as above and 
assume $G_0^{'}$ is the plane graph resulting from a fixed imbedding of $G_0$  by edge deletion.
Let $H$ be the plane graph resulting by replacing every  edge of the radial graph $\mathcal{R}(G_0^{'})$ with a digon. 
First color
the digons corresponding to edges in $\mathcal{R}(G_0^{'})$ with color $3$. The remaining faces of $H$  are quadrilaterals
$Q=xfx^{'}f^{'}x$  corresponding to edges $xx^{'}\in E(G_0^{'})$ that separate two faces $F$ and $F^{'}$ in $G_0^{'}$ corresponding to $f,f^{'}\in V(\mathcal{R}(G_0^{'}))$. Color these quadrilaterals with color $2$.
Let $G$  be the truncation of $H$; i.e., $G=Tr(H)$. Thus, $G$ is a $3-$connected cubic planar bipartite graph  
whose $3-$face-coloring has color set $\{1,2,3\}$; the $1-$faces of $G$ correspond to the vertices of $H$.

\begin{claim}\label{clm:3}
A set $L$ of edges in $G_0^{'}$ forms a hamiltonian path from $u$ to $v$ in $G_0^{'}$ if and
only if the set $\mathcal{T}$ of $2-$faces $($quadrilaterals$)$ in $H$ corresponding to the edges in $E(G_0^{'})\setminus L$  is a spanning tree of $2-$faces in $H$.
\end{claim}

Suppose $L$ is a hamiltonian path from $u$ to $v$ in $G_0^{'}$. Let $L^{'}= E(G_0^{'})\setminus L$, and let $\mathcal{T}$ be the corresponding quadrilaterals in $H$. Note that for any two edges $g,h\in L^{'}$, there is a sequence of edges $g = e_1,e_2, \ldots,e_k = h$ in $L^{'}$ such that each pair of
edges $e_i, e_{i+1}$ belongs to a face boundary, $1\le i\le k-1$. Therefore the $2-$faces in $\mathcal{T}$ induce a connected subgraph of $H$.
Notice also that every vertex in $H$ belongs to some face in $\mathcal{T}$, since every vertex $x\in V(G_0^{'})$ is
incident to an edge in $L^{'}$, and every face $F$ in $G_0^{'}$ has at least one edge in $L^{'}$.

  Finally,
the $2-$faces in $\mathcal{T}$ do not contain a cycle. Suppose to the contrary,  we had a cycle $Q_1Q_2\ldots Q_kQ_1$ of
$2-$faces in $\mathcal{T}$.
 Since the number of $2-$faces
in $\mathcal{T}$ containing $x$ is equal to $\deg_{G^{'}}(x)-\deg_L(x)=1$, for every vertex $x\in V(G_0^{'})$, so $Q_i$ and $Q_{i+1}$ share a vertex $f\in V(H)$  corresponding to  a face $F\in \mathcal{F}(G_0^{'})$. Thus 
$\{e_1, e_2, \ldots, e_k\}\subset L^{'}$, with $e_i$ corresponding to the face $Q_i$ in the cycle of $2-$faces in $\mathcal{T}$, separates
the graph $G_0^{'}$ into two components; so the hamiltonian path $L$ would have to contain at least
one of these edges $e_i\in L^{'}$, a contradiction. Therefore  
$\mathcal{T}$ is a spanning
tree of $2-$faces for $H$.

Conversely, suppose $\mathcal{T}$ is a spanning tree of  $2-$faces for $H$. Let $L^{'}$ be the corresponding edges in $G_0^{'}$, and let $L=E(G_0^{'})\setminus L^{'}$.
Each vertex $x\in V(G_0^{'})$  belongs to exactly one $2-$face $Q=xfx^{'}f^{'}x$ in $\mathcal{T}$, since every other 
$2-$face in $\mathcal{T}$ containing $x$ also contains either $f$ or $f^{'}$, and therefore these two $2-$faces share an edge 
 and thus cannot both be in the spanning tree of faces $\mathcal{T}$. Therefore every vertex in $G_0^{'}$ is incident to exactly one edge in $L^{'}$, and so the two vertices $u$ and $v$ of degree $2$ in $G_0^{'}$
are incident to exactly one edge in $L$, while the remaining vertices  of degree $3$ in $G_0^{'}$ are
incident to exactly two edges in $L$. That is, $L$ induces a path joining $u$ and $v$ in $G_0^{'}$ plus a possibly empty set
 of cycles in $G_0^{'}$, such that the path and the cycles are disjoint and cover all of $V(G_0^{'})$. We show that $L$ cannot contain a cycle in $G_0^{'}$, and so $L$ is just a hamiltonian path joining $u$
to $v$. 

Suppose
$L$ contains a cycle $C=e_1e_2\ldots e_ke_1$ in $G_0^{'}$. Let $F$ and $F^{'}$ be faces of $G_0^{'}$ inside and outside the
cycle of $C$, respectively, and let $f$ and $f^{'}$ be the vertices in $H$ corresponding to  $F$ and $F^{'}$, respectively.  Since $f$ and $f^{'}$ are vertices in the $\mathcal{T}$, there
is a unique sequence of $2-$faces $Q_1, Q_2,\ldots, Q_l$ in $\mathcal{T}$ such that $Q_1$ contains $f$, $Q_l$ contains $f^{'}$ and each
pair $Q_{i-1}$ , $Q_i$ share a vertex $f_i$ corresponding to 
a face in $G_0^{'}$, for $2\le i<l$. In particular, if we denote $f_1 = f$ and $f_l = f^{'}$, then for
some pair $f_i$, $f_{i+1}$ we must have for the 
corresponding  face $F_i$ in $G_0^{'}$ we must have $F_i\subseteq G_0^{'} \cap int(C)$ and for the corresponding  face $F_{i+1}$ in $G_0^{'}$ we must have $F_{i+1}\subseteq G_0^{'}\cap ext(C)$. This implies that the  $2-$face $Q_i$ in $\mathcal{T}$ corresponds to one of the edges $e_i$ in $L$ and not in $L^{'}$, a contradiction. This completes the proof of Claim~\ref{clm:3}.

 Therefore by Claim~\ref{clm:3}, $H$ has a
spanning tree of $2-$faces if and only if $G_0^{'}$ has a hamiltonian path from $u$ to $v$, and so the
question of whether $H$ has a spanning tree of $2-$faces is NP-complete.
}\end{proof}

We obtain two Corollaries from this result.

\begin{corollary}~$(${\rm\cite[Corollary~3]{Feder}}$)$\label{CR:3} 
Let $G$ be a $3-$connected cubic planar bipartite graph
with a $3-$face-coloring with color set $\{1,2,
3\}$, and let $H$ be the reduced graph obtained by contracting the $1-$faces. Suppose all vertices of $H$ have degree $8$. Then the question
of whether $H$ has an $A-$trail is NP-complete.
\end{corollary}

\begin{proof}{ Consider the reduced graph $H$ 
in the statement of Theorem~\ref{TH:4} where  all $2-$faces in $H$ are quadrilaterals,
corresponding to octagons in $G$. If we contract these $2-$faces, we obtain an $8-$regular reduced graph
$H^{'}$. By Theorem~\ref{PR:3}, $H^{'}$ has an $A-$trail  if and only
if $H$ has a spanning tree of $2-$faces, and this problem is NP-complete by Theorem~\ref{TH:4}.
}\end{proof}

\begin{corollary}~$(${\rm\cite[Corollary~4]{Feder}}$)$\label{CR:4} 
Let $G$ be a $3-$connected cubic planar bipartite graph
with a $3-$face-coloring with color set $\{1,2,
3\}$, and let $H_0$ be the reduced graph obtained by contracting the $1-$faces. Suppose that the $2-$faces in $H_0$ are octagons and digons and
the $3-$faces in $H_0$ are triangles. Then the question of whether $H_0$ has a spanning tree of faces is NP-complete.
\end{corollary}

\begin{proof}{  Let $H$ be the reduced graph of Theorem~\ref{TH:4}, with $2-$colored quadrilaterals and $3-$colored
digons. If $e$ and $f$ are the two parallel edges of a $3-$colored digon, 
subdivide $e$ with  vertex $w$  and 
subdivide $f$ with vertex $x$, with $w$ and $x$ joined by two parallel edges. The $3-$colored
digon splits thus into two $3-$color triangles and a $2-$colored digon, while the $2-$colored quadrilaterals become
$2-$colored octagons, in the new reduced graph $H_0$.

Suppose $H$ has a spanning tree of $2-$colored quadrilaterals $\mathcal{T}$. Select the corresponding $2-$colored
octagons in $H_0$. For a $3-$colored digon consisting of two edges $e$ and $f$ in $H$, if one of the two $2-$colored
quadrilaterals containing $e$ or $f$ is in $\mathcal{T}$, then select the 
$2-$colored digon joining the middle vertices
$w$ and $x$; if neither of the two $2-$colored quadrilaterals containing $e$ or $f$ is in $\mathcal{T}$, then select one
of the two $3-$colored triangles containing $w$ and $x$. The $2-$colored and 
$3-$colored faces in $H_0$ thus selected,
involving $2-$colored octagons, $2-$colored digons, and $3-$colored triangles, form a spanning tree of faces in $H_0$.

Conversely, suppose $H_0$ has a spanning tree of faces $\mathcal{T}_0$. Let $\mathcal{T}$ be the set of $2-$colored
quadrilaterals in $H$ such that the corresponding $2-$colored octagon is in $\mathcal{T}_0$. Note that for each
digon in $H$, only one of the corresponding two $3-$colored triangles and $2-$colored digon in $H_0$ can be in $\mathcal{T}_0$. Thus $\mathcal{T}$ is a spanning tree of $2-$colored faces.
Thus $H_0$ has a spanning tree of arbitrary faces if and only if $H$ has a spanning tree of
$2-$colored faces, and NP-completeness follows from Theorem~\ref{TH:4}.
}\end{proof}

\begin{lemma}\label{LE:G1}
If there exists a non-hamiltonian $3-$connected cubic planar bipartite graph, then there exists a hamiltonian $3-$connected cubic planar bipartite graph $G_1$
with a particular edge $e=uv$  such that $e\in E(C)$ for every hamiltonian cycle $C$ of $G_1$. Furthermore, if $e_1$ and $e_2$ are the two edges other than $e$ incident to $u$ in
$G_1$, then $G_1$ has a hamiltonian cycle $C_i$ traversing $e$ and $e_i$, for $i=1,2$.
\end{lemma}

\begin{proof}{
Suppose $G_0$ is a smallest counterexample to Barnette's Conjecture.

First we construct a hamiltonian $3-$connected cubic planar bipartite graph $G_1^{'}$
with a particular edge $e_0=u_0v$  such that $e_0\in E(C)$ for every hamiltonian cycle $C$ of $G_1^{'}$.

 Let $Q=wxyzw$ be a facial quadrilateral
 in $G_0$ and let $a_1$ be the third neighbour of $a$ in $G_0$, for $a\in \{w,x,y,z\}$.

Set $G_0^{'}=(G_0\setminus \{w,x,y,z\})\cup \{w_1x_1,y_1z_1\}$ and
$G_0^{''}=(G_0\setminus \{w,x,y,z\})\cup \{w_1z_1,x_1y_1\}$.
Both $G_0^{'}$ and $G_0^{''}$ are planar, cubic and bipartite.

Suppose that $G_0^{'}$ is   $3-$connected. By minimality of $G_0$, the graph $G_0^{'}$ has a hamiltonian
cycle. Furthermore, no hamiltonian cycle of $G_0^{'}$ goes through either the edge $w_1x_1$ or the edge $y_1z_1$, otherwise,  
 we can extend this cycle to a hamiltonian cycle  in $G_0$, a contradiction.
 
  We have thus guaranteed that no
hamiltonian cycle in $G_1^{'}=G_0^{'}$ traverses a particular edge $w_1x_1$, and thus every hamiltonian cycle traverses an edge $e_0$
adjacent to $w_1x_1$, as desired.

Suppose instead that $G_0^{'}$ and $G_0^{''}$ are both  $2-$connected only. Then there are two edge cuts of size four $T_1$ and $T_2$ in $G_0$ such that $\{w_1x_1,y_1z_1\}\subset T_1$ and  $\{w_1z_1,x_1y_1\}\subset T_2$.

 Removing the vertices $w,x,y,z$ and the two edge cuts $T_1$ and $T_2$ 
separates $G_0$ into four components $R_1, R_2, R_3, R_4$, with the removed edges of $G_0$ including
an edge from $R_i$ to $R_{i+1}$, for $i=1,2,3$, and an edge from 
$R_4$ to $R_1$, plus the four edges from the four $R_i$'s  adjacent to $Q$.

That is, each $R_i$ has three   edges whose 
endvertices not in $R_i$ can be identified to a single vertex $r_i$ to construct $R_i^{'}$, since their
three endvertices in $R_i$ are at even distance from each other (in the $2-$vertex-coloring of $G_0$, the three 
$2-$colored vertices of $R_i,\ 1\le i\le 4$, must have the same color;
otherwise, two copies of such $R_i$ could be used to construct a cubic bipartite graph having a bridge). Clearly, $R_i^{'}$
is $3-$connected, cubic, planar, and bipartite, for each $i=1,\ldots,4$.

By minimality of $G_0$ each such   $R_i^{'}$
 has a hamiltonian cycle, yet it is not the case that each of the
three choices of two edges going into each $R_i$ yields a hamiltonian cycle, since otherwise we
would obtain a hamiltonian cycle for $G_0$. Thus one of the three edges
incident to $r_i$ in $R_i^{'}$ must belong to every hamiltonian cycle, thus yielding a $3-$connected cubic
planar bipartite graph $G_1^{'}$ with an edge $e_0=u_0v$ that belongs to every hamiltonian cycle of $G_1^{'}$.

It remains to ensure that a hamiltonian cycle in $G_1$, which is forced to take $e =uv$,
can take either $e_1$ or $e_2$ out of $u$. Suppose instead that every hamiltonian cycle in $G_1^{'}$ is forced to
take $e_1^{'}=u_0v^{'}\in E(G_1^{'})$ as well. 

Consider the $3-$connected cubic planar bipartite graph $Q_3$ of a cube with
$8$ vertices and $t\in V(Q_3)$ and $N_{Q_3}(t)=\{u,u^{'},u^{''}\}$.
Let $G_1=(G_1^{'}\setminus \{u_0\})\cup (Q_3\setminus \{t\})\cup\{uv,u^{'}v^{'},u^{''}v^{''}\}$, where  
$e_2^{'}=u_0v^{''}\in E(G_1^{'})$.

It is easy to check that $G_1$ is hamiltonian and every hamiltonian cycle 
going through $e=uv$; and furthermore, $G_1$ has a hamiltonian cycle $C_i$ going through $e$ and $e_i$, for $i=2,3$, where $e_1$ and $e_2$ are the two edges other than $e$ incident to $u$ in
$G_1$.
 }\end{proof}

\begin{theorem}~$(${\rm\cite[Theorem~5]{Feder}}$)$\label{TH:5} 
Assume that  Barnette's Conjecture is false. Then the question of whether a $3-$connected cubic planar bipartite graph  has
a hamiltonian cycle, is NP-complete.
\end{theorem}

\begin{proof}{ Takanori et al.~\cite{Takanori} showed that the question of whether a $2-$connected cubic
planar bipartite graph $R$ has a hamiltonian cycle is NP-complete.

 If such an $R$ has a $2-$edge-cut 
$\{e_1,e_2\}$ that separates $R$ into two components $R^{'}$ and $R^{''}$, then their endpoints in either side
are at odd distance (see the above argument),
 so we may instead join the two endpoints of $e_1$ and $e_2$ in $R^{'}$ and $R^{''}$, separately, and
ask whether $R^{'}$ and $R^{''}$ both contain a hamiltonian cycle containing the added edge joining
the endpoints of $e_1$ and $e_2$.

Repeating this decomposition process, we eventually reduce the
question of whether $R$ has a hamiltonian cycle to the question of whether various $R_i$'s each
contain a hamiltonian cycle going through certain prespecified edges, with each $R_i$ being
$3-$connected. Thus the question of whether a $3-$connected cubic planar bipartite graph $G^{'}$
has a hamiltonian cycle going through certain prespecified edges is NP-complete.

 Let a $3-$connected cubic planar bipartite graph $G^{'}$
 with certain prespecified edges $e^{'}_1,\ldots,e^{'}_k$ that a hamiltonian cycle must traverse, be given. 
Assume that  $e_i^{'}=x_iy_{i,1}$ and $N_{G^{'}}(x_i)=\{y_{i,1},y_{i,2},y_{i,3}\}$, for $i=1,\ldots,k$.

Suppose that  Barnette's Conjecture is false. Then by 
Lemma~\ref{LE:G1}, there exists a hamiltonian $3-$connected cubic planar bipartite graph $G_{i}$ with a vertex $u_i\in V(G_{i})$ and $N_{G_{i}}(u_i)=\{v_{i,1},v_{i,2},v_{i,3}\}$
 such that  every hamiltonian cycle
in $G_{i}$ traverses $e_i= u_iv_{i,1}\in E(G_i)$, $i=1\ldots,k$. Furthermore,   $G_i$ has a hamiltonian cycle traversing $e$ and $u_iv_{i,j}$, for $i=1,\ldots,k$ and $j=2,3$.

Construct a new $3-$connected cubic planar bipartite graph 
$$G=\bigg(G^{'}\setminus \{x_1,\ldots,x_k\}\bigg)\cup 
\bigg(\bigcup_{i=1}^k(G_i\setminus \{u_i\})\bigg) \cup \bigg(\bigcup_{i=1}^k\{v_{i,1}y_{i,1},v_{i,2}y_{i,2},v_{i,3}y_{i,3}\}\bigg).$$

Since every hamiltonian cycle in  $G_i$ traversing the edge $e_i$ and   $G_i$ has also a hamiltonian cycle traversing $e_i$ and $u_iv_{i,j}$, for $i=1,\ldots,k$ and $j=2,3$, the resulting graph $G$ has a hamiltonian cycle if and only if $G^{'}$ 
has a hamiltonian cycle traversing the edges $e_1^{'},\ldots,e_k^{'}$.
 Therefore, whether the resulting $3-$connected cubic
planar bipartite graph $G$ has a hamiltonian cycle, is NP-complete.
}\end{proof}

\end{document}